\documentclass[11pt,reqno]{amsart}
\usepackage{amsmath}
\usepackage{amssymb}
\usepackage{amsthm}
\usepackage{mathrsfs}
\usepackage{enumerate}
\usepackage[english]{babel}
\usepackage{hyperref}
\usepackage[all,cmtip]{xy}
\entrymodifiers={+!!<0pt,\fontdimen22\textfont2>}

\newtheorem{thm}{Theorem}[section]
\newtheorem{lem}[thm]{Lemma}
\newtheorem*{thm-i}{Theorem}
\newtheorem{cor}[thm]{Corollary}

\theoremstyle{definition}

\newtheorem{rem}[thm]{Remark}
\newtheorem{defn}[thm]{Definition}

\newtheorem{ex}[thm]{Example}

\newtheorem*{acknowledgments*}{Acknowledgments}

\numberwithin{equation}{section}

\theoremstyle{remark}

\usepackage[a4paper,text={15.5cm,26cm},centering]{geometry}
\mathchardef\ordinarycolon\mathcode`\: 
\def\vcentcolon{\mathrel{\mathop\ordinarycolon}} 
\providecommand*\coloneqq{\mathrel{\vcentcolon\mkern-1.2mu}=}

\def\C{{\mathbb C}} 
\def\N{{\mathbb N}}

\DeclareMathOperator\x{\otimes}

\DeclareMathOperator\Tor{Tor}

\def\cast{$C^{*}$}

\def\id{{\mathrm{id}}}

\begin{document}

\title{$K$-continuity is equivalent to $K$-exactness}
\author{Otgonbayar Uuye}
\date{\today}
\address{
School of Mathematics\\ 
Cardiff University\\
Senghennydd Road\\
Cardiff, Wales, UK.\\
CF24 4AG}
\email{UuyeO@cardiff.ac.uk}
\keywords{Operator algebras, $K$-theory, $K$-exactness, $K$-continuity}
\subjclass[2010]{Primary (46L80); Secondary (46L06)}

\begin{abstract} It is well known that the functor of taking the minimal tensor product with a fixed \cast-algebra preserves inductive limits if and only if it preserves extensions. In other words, tensor continuity is equivalent to tensor exactness. We consider a $K$-theoretic analogue of this result  and show that $K$-continuity is equivalent to $K$-exactness, using a result of M.~D{\u{a}}d{\u{a}}rlat.
\end{abstract}

\maketitle

\section{Introduction}

We denote the {\em spatial} or {\em minimal} tensor product of \cast-algebras by the symbol $\otimes$ (cf.\ \cite[Section IV.4]{MR1873025}, \cite[Section 3.3]{MR2391387}).

Let $A$ be a \cast-algebra. We say that $A$ is {\em $\otimes$-exact} if for every extension (i.e.\ short exact sequence)
	\begin{equation*}
	\xymatrix@1{0 \ar[r] & I \ar[r]  & D \ar[r] & B \ar[r] & 0}
	\end{equation*} 
of \cast-algebras, the natural sequence 
	\begin{equation*}
	\xymatrix@1{0 \ar[r] & A \otimes I \ar[r]  & A \otimes D \ar[r] & A \otimes B \ar[r] & 0}
	\end{equation*} 
is exact. Let $M_{n}$ denote the \cast-algebra of $n \times n$ complex matrices. Letting 
	\begin{align*}
	\prod_{n = 0}^{\infty} M_{n} &\coloneqq \left\{(a_{n})_{n = 0}^{\infty} \,\Big\vert\, a_{n} \in M_{n} \text{ for all $n$ and }  \sup_{n} ||a_{n}|| < \infty \right\}\quad\text{and}\\
	\bigoplus_{n = 0}^{\infty} M_{n} &\coloneqq \left\{(a_{n})_{n = 0}^{\infty} \,\Big\vert\,  a_{n} \in M_{n} \text{ for all $n$ and }  \lim_{n \to \infty} ||a_{n}|| = 0 \right\},	
	\end{align*}
we get an extension
	\begin{equation*}
	\xymatrix{
	0 \ar[r] & \bigoplus_{n=0}^{\infty}M_{n} \ar[r]  & \prod_{n=0}^{\infty} M_{n} \ar[r] &  \prod_{n=0}^{\infty} M_{n}\big/\bigoplus_{n=0}^{\infty} M_{n} \ar[r] & 0}.
	\end{equation*}
E.~Kirchberg proved the following fundamental result about $\otimes$-exactness. See \cite{MR2391387} for more details.
\begin{thm-i}[{E.~Kirchberg \cite{MR715549,MR1322641}}] Let $A$ be a \cast-algebra. The following statements are equivalent.
\begin{enumerate}[(i)]
\item\label{it t-exact} The algebra $A$ is $\otimes$-exact.
\item\label{it m-exact} The sequence 
	\begin{equation*}
	\xymatrix@C=12pt{
	0 \ar[r] & A \otimes \bigoplus_{n=0}^{\infty}M_{n} \ar[r]  & A \otimes \prod_{n=0}^{\infty} M_{n} \ar[r] & A \otimes \left(\prod_{n=0}^{\infty} M_{n}\big/\bigoplus_{n=0}^{\infty} M_{n}\right) \ar[r] & 0
	}
	\end{equation*}
is exact.
\item\label{it nuc emb} The algebra $A$ is nuclearly embeddable in the sense of \cite{MR1026768}.
\qed
\end{enumerate}
\end{thm-i}
We remark that the implication (\ref{it nuc emb}) $\Rightarrow$ (\ref{it t-exact}) was proved by S.~Wassermann in \cite{MR1058315}.

We say that $A$ is {\em $\otimes$-continuous} if for every inductive sequence  
	\begin{equation*}
	\xymatrix@1{B_{0} \ar[r] & B_{1} \ar[r] & B_{2} \ar[r] & \dots}
	\end{equation*}
of \cast-algebras, the natural (surjective) map 
	\begin{equation*}
	\xymatrix@1{\varinjlim (A \otimes B_{n}) \ar[r] & A \otimes \varinjlim B_{n}}
	\end{equation*} 
is an isomorphism, where $\varinjlim$ denotes the inductive limit functor. 

The following result is well-known and follows from the equivalence (\ref{it t-exact}) $\Leftrightarrow$ (\ref{it m-exact}) in the theorem above. N.~Ozawa attributes it to E.~Kirchberg. 
\begin{thm}\label{thm K}
A \cast-algebra is $\otimes$-exact if and only if it is $\otimes$-continuous.
\end{thm}

In this paper, we consider a $K$-theoretic analogue of this result. See \cite{MR1222415, MR1656031,MR1783408} for details about topological $K$-theory for \cast-algebras. We say that a \cast-algebra $A$ is {\em $K$-exact} for every extension
	\begin{equation*}
	\xymatrix@1{0 \ar[r] & I \ar[r]  & D \ar[r] & B \ar[r] & 0}
	\end{equation*} 
 of \cast-algebras, the sequence
	\begin{equation*}
	\xymatrix@1{K_{0}(A \otimes I) \ar[r]  & K_{0}(A \otimes D) \ar[r] & K_{0}(A \otimes B)}
	\end{equation*}
is exact in the middle. We say that a \cast-algebra $A$ is {\em $K$-continuous} if for every inductive sequence 
	\begin{equation*}
	\xymatrix@1{B_{0} \ar[r] & B_{1} \ar[r] & B_{2} \ar[r] & \dots}
	\end{equation*}
of \cast-algebras, the natural map 
	\begin{equation*}
	\varinjlim K_{0}(A \otimes B_{i}) \longrightarrow K_{0}(A \otimes \varinjlim B_{i})
	\end{equation*} 
is an isomorphism.

The following is our main result.
\begin{thm}\label{thm main} A \cast-algebra is $K$-exact  if and only if it is $K$-continuous.
\end{thm}

In section~\ref{sec K}, we give a proof of Theorem~\ref{thm K} as we couldn't find a direct reference and the proof of the implication Theorem~\ref{thm K}($\Rightarrow)$ is used in proof of Theorem~\ref{thm main}($\Rightarrow$). In section~\ref{sec main}, we study the notions of $K$-exactness and $K$-continuity and prove Theorem~\ref{thm main}. We note that our proof of the implication Theorem~\ref{thm main}($\Leftarrow$) uses \cite[Theorem 3.11]{MR1262931} in a crucial way.
\begin{acknowledgments*} We thank Takeshi Katsura and Kang Li for interesting conversations on the topic of the paper. The author is an EPSRC postdoctoral fellow.
\end{acknowledgments*}

\section{Proof of Theorem~\ref{thm K}}\label{sec K}
Let $\N \coloneqq \{0, 1, \dots\}$ denote the set of positive integers. The following is obvious.
\begin{lem}\label{lem limit inj}
Let $A$ be a \cast-algebra and let 
	\begin{equation*}
	\xymatrix@1{B_{0} \ar[r] & B_{1} \ar[r] & B_{2} \ar[r] & \dots}
	\end{equation*}
be an inductive sequence of \cast-algebras. If the connecting maps are all injective, then the map  
	\begin{equation*}
	\xymatrix@1{\varinjlim (A \otimes B_{n}) \ar[r] & A \otimes \varinjlim B_{n}}
	\end{equation*}
is an isomorphism. \qed	 
\end{lem}

\begin{lem}\label{lem main} Consider an inductive sequence of extensions of \cast-algebras
	\begin{equation*}
	\xymatrix{
	0 \ar[r] & I_{n} \ar[r] \ar[d]^{\imath_{n}} & D_{n} \ar[r] \ar[d]^{\delta_{n}}& B_{n} \ar[r] \ar[d]^{\beta_{n}}& 0\phantom{.}\\
	0 \ar[r] & I_{n+1} \ar[r]  & D_{n+1} \ar[r] & B_{n+1} \ar[r] & 0.\\
	}
	\end{equation*}
\begin{enumerate}[(i)]
\item\label{it lim diag} The limit sequence
	\begin{equation*}
	\xymatrix{
	0 \ar[r] & \varinjlim I_{n} \ar[r]  & \varinjlim D_{n} \ar[r] & \varinjlim B_{n} \ar[r] & 0
	}
	\end{equation*}
is  exact.
\item\label{it lim diag tensor} Suppose that for every $n \in \N$, the extension
	\begin{equation}\label{eq ext_n}
	\xymatrix{
	0 \ar[r] & I_{n} \ar[r]  & D_{n} \ar[r] & B_{n} \ar[r] & 0
	}
	\end{equation}
is split\footnote{Locally split is enough for our purposes \cite[Proposition 3.7.6]{MR2391387}. See \cite{MR791294}.} (i.e.\ the quotient map admits a $\ast$-homomorphic section) and the connecting maps $\imath_{n}$ and $\delta_{n}$ are injective. Then for any \cast-algebra $A$, the map
	\begin{equation*}
	\xymatrix@1{\varinjlim (A \otimes B_{n}) \ar[r] & A \otimes \varinjlim B_{n}}
	\end{equation*}
is an isomorphism if and only if the 	sequence
	\begin{equation*}
	\xymatrix{
	0 \ar[r] & A \otimes \varinjlim I_{n} \ar[r]  & A \otimes \varinjlim D_{n} \ar[r] & A \otimes \varinjlim B_{n} \ar[r] & 0
	}
	\end{equation*}
is exact.	
\end{enumerate}
\end{lem}
\begin{proof} (\ref{it lim diag}) is clear. For (\ref{it lim diag tensor}), consider the diagram
	\begin{equation}\label{eq diag of ses}
	\xymatrix{
	0 \ar[r] & \varinjlim (A \otimes I_{n}) \ar[r] \ar[d]^{\imath} & \varinjlim (A \otimes D_{n}) \ar[r] \ar[d]^{\delta} & \varinjlim (A \otimes B_{n}) \ar[r] \ar[d]^{\beta} & 0\phantom{.}\\
	0 \ar[r] & A \otimes \varinjlim I_{n} \ar[r]  & A \otimes \varinjlim D_{n} \ar[r] & A \otimes \varinjlim B_{n} \ar[r] & 0.\\
	}
	\end{equation}	
Since the connecting maps $\imath_{n}$ and $\delta_{n}$ are injective, the maps $\imath$ and $\delta$ are isomorphisms by Lemma~\ref{lem limit inj}. For any $n \in \N$, since (\ref{eq ext_n}) is split exact, the sequence
	\begin{equation*}
	\xymatrix{
	0 \ar[r] & A \otimes I_{n} \ar[r]  & A \otimes D_{n} \ar[r] & A \otimes B_{n} \ar[r] & 0
	}
	\end{equation*}
is exact. Thus the top row is exact by (\ref{it lim diag}). It follows that $\beta$ is an isomorphism if and only if the bottom row is exact by five-lemma.
\end{proof}

\begin{proof}[Proof of Theorem~\ref{thm K}] 
($\Rightarrow$): Let $A$ be a $\otimes$-exact \cast-algebra and let $\xymatrix@1{B_{0} \ar[r]^{\beta_{0}} & B_{1} \ar[r]^{\beta_{1}}  & \dots}$ be an inductive sequence.

Let $n \in \N$ and let $I_{n} \coloneqq \bigoplus_{k =0}^{n-1} B_{k}$ and let $D_{n} \coloneqq \bigoplus_{k =0}^{n} B_{k}$. Then the obvious inclusion and projection maps give a split extension
	\begin{equation*}
	\xymatrix{
	0 \ar[r] & I_{n} \ar[r]  & D_{n} \ar[r] & B_{n} \ar[r] & 0.
	}
	\end{equation*}
Let $\imath_{n}\colon I_{n} \to I_{n+1}$ denote the natural inclusion and let $\delta_{n}\colon D_{n} \to D_{n+1}$ denote the injective map given by
	\begin{equation*}
	\xymatrix@R=0pt{
	B_{0} \ar[r]^{\id} & B_{0}\\
	\oplus & \oplus\\
	\vdots  \ar[r]^{\id}& \vdots \\
	\oplus & \oplus\\
	B_{n} \ar[r]^{\id} \ar[rdd]_-{\beta_{n}}& B_{n}\\
	 & \oplus\\
	  & B_{n+1}.\\
	}
	\end{equation*}	
Then we get a map of extensions 
	\begin{equation*}
	\xymatrix{
	0 \ar[r] & I_{n} \ar[r] \ar[d]^{\imath_{n}} & D_{n} \ar[r] \ar[d]^{\delta_{n}}& B_{n} \ar[r] \ar[d]^{\beta_{n}}& 0\phantom{.}\\
	0 \ar[r] & I_{n+1} \ar[r]  & D_{n+1} \ar[r] & B_{n+1} \ar[r] & 0.\\
	}
	\end{equation*}
Since $A$ is $\otimes$-exact, the sequence
	\begin{equation*}
	\xymatrix{
	0 \ar[r] & A \otimes \varinjlim I_{n} \ar[r]  & A \otimes \varinjlim D_{n} \ar[r] & A \otimes \varinjlim B_{n} \ar[r] & 0
	}
	\end{equation*}
is exact, hence the map 
	\begin{equation*}
	\varinjlim (A \otimes B_{n}) \to A \otimes \varinjlim B_{n}
	\end{equation*} 
is an isomorphism by Lemma~\ref{lem main}(\ref{it lim diag tensor}).

($\Leftarrow$): Conversely, let $A$ be a $\otimes$-continuous \cast-algebra. Let $M_{n}$, $n \in \N$, denote the \cast-algebra of $n \times n$ complex matrices. Consider the following inductive system of split extensions
	\begin{equation*}
	\xymatrix{
	0 \ar[r] & \bigoplus_{k=0}^{n}M_{k} \ar[r] \ar[d]& \prod_{k=0}^{\infty}M_{k} \ar[r] \ar@{=}[d] & \prod_{k = n+1}^{\infty} M_{k}\ar[r] \ar[d] & 0 \\
	0 \ar[r] & \bigoplus_{k=0}^{n+1}M_{k} \ar[r] & \prod_{k=0}^{\infty}M_{k} \ar[r] & \prod_{k = n+2}^{\infty} M_{k}\ar[r] & 0 
	}
	\end{equation*}
given by the obvious injection and projection maps. Since $A$ is $\otimes$-continuous, the map
	\begin{equation*}
	\varinjlim \left(A \otimes \prod_{k=n+1}^{\infty} M_{k}\right) \to A \otimes \left(\varinjlim \prod_{k=n+1}^{\infty} M_{k}\right)
	\end{equation*}
is an isomorphism, hence the sequence
	\begin{equation*}
	\xymatrix{
	0 \ar[r] & A \otimes \bigoplus_{n=0}^{\infty}M_{n} \ar[r]  & A \otimes \prod_{n=0}^{\infty} M_{n} \ar[r] & A \otimes \left(\prod_{n=0}^{\infty} M_{n}\big/\bigoplus_{n=0}^{\infty} M_{n}\right) \ar[r] & 0
	}
	\end{equation*}
is exact by Lemma~\ref{lem main}(\ref{it lim diag tensor}). It follows that $A$ is $\otimes$-exact (cf.\ \cite{MR715549}).
\end{proof}

\section{$K$-exactness and $K$-continuity}\label{sec main}

\subsection{$K$-exactness}
\begin{defn} We say that a \cast-algebra $A$ is {\em $K$-exact} if for every extension 
	\begin{equation*}
	\xymatrix@1{0 \ar[r] & I \ar[r]  & D \ar[r] & B \ar[r] & 0}
	\end{equation*} 
of \cast-algebras, the sequence
	\begin{equation*}
	\xymatrix@1{K_{0}(A \otimes I) \ar[r]  & K_{0}(A \otimes D) \ar[r] & K_{0}(A \otimes B)}
	\end{equation*}
is exact in the middle.	
\end{defn}

\begin{rem}\label{rem six-term} A \cast-algebra $A$ is $K$-exact if and only if for every extension 
	\begin{equation*}
	\xymatrix@1{0 \ar[r] & I \ar[r]  & D \ar[r] & B \ar[r] & 0}
	\end{equation*} 
of \cast-algebras, the natural six-term sequence
	\begin{equation*}
	\xymatrix{K_{0}(A \otimes I) \ar[r]  & K_{0}(A \otimes D) \ar[r] & K_{0}(A \otimes B) \ar[d]\\
	K_{1}(A \otimes B) \ar[u] & \ar[l] K_{1}(A \otimes D) & \ar[l] K_{1}(A \otimes I)
	}
	\end{equation*}
is exact (cf.\ \cite[Theorem 21.4.4]{MR1656031}).	
\end{rem}

\begin{ex}
$\otimes$-exact \cast-algebras are $K$-exact, by the half-exactness of $K$-theory (cf.\ \cite[Theorem 6.3.2]{MR1222415}, \cite[Theorem 5.6.1]{MR1656031}). 
\end{ex}

\begin{defn}
Let $C_{0}[0, 1)$ denote the commutative \cast-algebra of continuous functions on the interval $[0, 1]$ vanishing at $1 \in [0, 1]$, 
and let 
	\begin{equation*}
	\mathrm{ev}_{0}\colon \xymatrix@R=0pt{C_{0}[0, 1) \ar[r] & \C\\ f \ar@{|->}[r]& f(0)}
	\end{equation*}
denote the evaluation map at $0 \in [0, 1)$.
\end{defn}
  
\begin{defn} Let $\phi\colon D \to B$ be a $*$-homomorphism of \cast-algebras. The {\em mapping cone} $C_{\phi}$ of $\phi$ is given by the pullback
	\begin{equation*}
	\xymatrix{C_{\phi} \ar[r] \ar[d] & C_{0}[0, 1) \x B \ar[d]^{\mathrm{ev}_{0} \x \id_{B}}\\
	D \ar[r]^{\phi} & B}
	\end{equation*}
\end{defn}

\begin{rem}\label{rem mapping cone} Let $\phi\colon D \to B$ be a $\ast$-homomorphism of \cast-algebras. Then for any \cast-algebra $A$, there is a natural isomorphism $C_{\id_{A} \otimes \phi} \cong A \otimes C_{\phi}$. Indeed, since $\mathrm{ev}_{0}\colon C_{0}[0, 1) \to \C$ is admits a completely positive section, we have a map of extensions 
	\begin{equation*}
\xymatrix{
	0 \ar[r] & A \otimes C_{0}(0, 1) \otimes B \ar[r] \ar@{=}[d] & A \otimes C_{\phi} \ar[r] \ar[d] & A \otimes D \ar[r] \ar[d]^{\id_{A} \otimes \phi}& 0\phantom{.}\\
	0 \ar[r] & A \otimes C_{0}(0, 1) \otimes B  \ar[r]  & A \otimes C_{0}[0, 1) \otimes B \ar[r] & A \otimes  B \ar[r] & 0,\\
	}
	\end{equation*}
where $C_{0}(0, 1)$ denotes the kernel of $\mathrm{ev}_{0}$. Now it is easy to see that the square on the right is a pullback square.	
\end{rem}


\begin{lem}[{cf.\ \cite[p.\ 335-336]{MR1911663}.}]\label{lem k-exact cone} A \cast-algebra $A$ is $K$-exact if and only if for every extension 
	\begin{equation*}
	\xymatrix@1{0 \ar[r] & I \ar[r]  & D \ar[r]^{q} & B \ar[r] & 0}
	\end{equation*}
of {\em separable} \cast-algebras, the natural inclusion map $\xymatrix@1{\iota\colon I \ar[r] & C_{q}}$ induces an isomorphism 
	\begin{equation*}
	(\id_{A} \otimes \iota)_{*}\colon K_{0}(A \otimes I) \cong K_{0}(A \otimes C_{q}).
	\end{equation*}
\end{lem}
\begin{proof}
Let $A$ be a \cast-algebra and let 
	\begin{equation}\label{eq ext}
	\xymatrix@1{0 \ar[r] & I \ar[r]  & D \ar[r]^{q} & B \ar[r] & 0}
	\end{equation}
be an extension of (not necessarily separable) \cast-algebras.

 ($\Rightarrow$): Suppose that $A$ is $K$-exact. By the homotopy invariance of $K$-theory, we have  $K_{*}(A \otimes C_{0}[0, 1) \otimes B) = 0$. Hence, applying Remark~\ref{rem six-term} to the pullback extension
	\begin{equation*}
	\xymatrix@1{0 \ar[r] & I \ar[r]^{\imath}  & C_{q} \ar[r] & C_{0}[0, 1) \otimes B \ar[r]& 0},
	\end{equation*}
we see that $\id_{A} \otimes \imath$ induces an isomorphism $K_{0}(A \otimes I) \cong K_{0}(A \otimes C_{q})$.

($\Leftarrow$): Conversely, suppose that $A$ satisfies the necessary condition in the lemma. We prove that the sequence
	\begin{equation*}
	\xymatrix@1{K_{0}(A \otimes I) \ar[r]  & K_{0}(A \otimes D) \ar[r] & K_{0}(A \otimes B)}
	\end{equation*}
is exact in the middle.

If $I$, $D$ and $B$ are separable, then the exactness follows from the Puppe exact sequence (cf.\ \cite{MR658514} or \cite[Lemma 6.4.8]{MR1222415}) and the natural isomorphism $C_{\id_{A} \otimes q} \cong A \otimes C_{q}$ of Remark~\ref{rem mapping cone}. 

The general case is reduced to the separable case as follows. Let $\Lambda$ denote the set of separable \cast-subalgebras of $D$, ordered by inclusion. Then $\Lambda$ is a directed set. For each $E \in \Lambda$, we associate a subextension
	\begin{equation*}
	\xymatrix{0 \ar[r] & I \cap E \ar[r] \ar@{^{(}->}[d] & E \ar[r] \ar@{^{(}->}[d] & E/(I \cap E) \ar[r] \ar@{^{(}->}[d]& 0\phantom{.}\\
	0 \ar[r] & I \ar[r]  & D \ar[r] & B \ar[r] & 0.}
	\end{equation*}
It is clear that the inductive limit of the subextensions is the extension (\ref{eq ext}). Since all the connecting maps are injective, the proof is complete by the continuity of $K$-theory (cf.\ \cite[Proposition 6.2.9]{MR1222415}) and the exactness of the inductive limit functor for abelian groups (cf.\ \cite[Theorem 2.6.15]{MR1269324}).
\end{proof}

\begin{cor}\label{cor factor} A \cast-algebra $A$ is $K$-exact if and only if the functor $B \mapsto K_{0}(A \otimes B)$, from the category of {\em separable} \cast-algebras to abelian groups, factors through the category $E$ of Higson (cf.\ \cite{MR1068250, MR1065438}).
\end{cor}
\begin{proof} Let $A$ be a \cast and let $F(B) \coloneqq K_{0}(A \otimes B)$.

($\Rightarrow$): Suppose that $A$ is $K$-exact. Then $F$ is half-exact and since $F$ is homotopy invariant and stable (under tensoring with the compacts), it factors through the category $E$ by the universal property (cf.\ \cite[Th\'eor\`eme 7]{MR1065438}).

($\Leftarrow$): Suppose that $F$ factors through $E$. For any extension $\xymatrix@1{0 \ar[r] & I \ar[r]  & D \ar[r]^{q} & B \ar[r] & 0}$ of separable \cast-algebras, the inclusion $\xymatrix@1{\iota\colon I \ar[r] & C_{q}}$ is an equivalence in $E$ (cf.\ \cite[Lemma 12]{MR1065438}). Now Lemma~\ref{lem k-exact cone} completes the proof.
\end{proof}

%
%
%

\begin{ex}
\cast-algebras satisfying the K\"unneth formula of Schochet (cf.\ \cite{MR650021}, \cite[Theorem 23.1.3]{MR1656031}) are $K$-exact by Lemma~\ref{lem k-exact cone} (cf.\ \cite[Remark 4.3]{MR2100669}). In particular, the full group \cast-algebra $C^{*}(F_{2})$ of the free group $F_{2}$ on two generators is $K$-exact (but not $\otimes$-exact). 
\end{ex}

Needless to say, not all \cast-algebras are $K$-exact.
\begin{ex}\label{ex K-exactness} 
\begin{enumerate} 
\item\label{item CG} Let $\Gamma$ be an infinite countable discrete group with Khazdan property (T), Kirchberg property (F) and Akemann-Ostrand property (AO), such as a lattice in $\mathrm{Sp}(n, 1)$ (cf.\ \cite{MR2562137}). The {\em full} group \cast-algebra $C^{*}(\Gamma)$ is not $K$-exact (G.~Skandalis \cite{MR1143449}).
\item\label{item M} The product $\prod_{n = 0}^{\infty} M_{n}$ is not {\em $K$-exact} (N.~Ozawa \cite[Theorem A.1]{MR1964549}). 
\end{enumerate}
\end{ex}

\subsection{$K$-continuity}
\begin{defn}\label{defn K-cont} We say that a \cast-algebra $A$ is {\em $K$-continuous} if for every inductive sequence
	\begin{equation*}
	\xymatrix@1{B_{0} \ar[r] & B_{1} \ar[r] & B_{2} \ar[r] & \dots}
	\end{equation*}
of \cast-algebras, the natural map 
	\begin{equation*}
	\varinjlim K_{0}(A \otimes B_{i}) \longrightarrow K_{0}(A \otimes \varinjlim B_{i})
	\end{equation*} 
is an isomorphism.
\end{defn}

\begin{rem} In Definition~\ref{defn K-cont}, we could use $K_{1}$ instead of $K_{0}$. 
\end{rem}
\begin{ex} $\otimes$-continuous \cast-algebras are $K$-continuous, by the continuity of $K$-theory (cf.\ \cite[Proposition 6.2.9]{MR1222415}, \cite[5.2.4, 8.1.5]{MR1656031}).
\end{ex}

\begin{ex} \cast-algebras satisfying the K\"unneth formula of Schochet (cf.\ \cite{MR650021}, \cite[Theorem 23.1.3]{MR1656031}) are $K$-continuous. Indeed, let $A$ be a \cast-algebra satisfying the K\"unneth formula and let 
	\begin{equation*}
	\xymatrix@1{B_{0} \ar[r] & B_{1} \ar[r] & B_{2} \ar[r] & \dots}
	\end{equation*}
be an inductive sequence of \cast-algebras. Then the top row in the diagram
	\begin{equation*}
	\xymatrix{
	0 \ar[r] & \varinjlim K_{*}(A) \otimes K_{*}(B_{i}) \ar[r]  \ar[d] & \varinjlim K_{*}(A \otimes B_{i}) \ar[r] \ar[d] & \varinjlim \Tor(K_{*}(A), K_{*}(B_{i})) \ar[d] \ar[r] & 0\\
	0 \ar[r] & K_{*}(A) \otimes K_{*}(\varinjlim  B_{i}) \ar[r]  & K_{*}(A \otimes \varinjlim B_{i}) \ar[r] & \Tor(K_{*}(A), K_{*}(\varinjlim  B_{i})) \ar[r] & 0
	}
	\end{equation*}
is an extension of abelian groups by \cite[Theorem 2.6.15]{MR1269324}  and the second row is an extension by the K\"unneth formula.	 The left and right vertical maps are isomorphisms by \cite[Corollary 2.6.17]{MR1269324} and thus the middle vertical map is also an isomorphism by five-lemma. Hence $A$ is $K$-continuous.
\end{ex}

\subsection{Proof of Main Theorem~\ref{thm main}}
\begin{proof}[Proof of Theorem~\ref{thm main}] 
($\Rightarrow$):
Let $A$ be a $K$-exact \cast-algebra and let 
	\begin{equation*}
	\xymatrix@1{B_{0} \ar[r]^{\beta_{0}} & B_{1} \ar[r]^{\beta_{1}}  & \dots}
	\end{equation*}
be an inductive sequence. We use the notations of the proof of Theorem~\ref{thm K}($\Rightarrow$). Applying $K$-theory to the diagram (\ref{eq diag of ses}), we get a map of {\em exact} sequences
	\begin{equation*}
	\xymatrix@C=6pt{
	K_{0}(\varinjlim A \otimes I_{n}) \ar[r] \ar[d]^{\cong} & K_{0}(\varinjlim A \otimes D_{n}) \ar[r] \ar[d]^{\cong} & K_{0}(\varinjlim A \otimes B_{n}) \ar[r] \ar[d]^{\beta_{*}} & K_{1}(\varinjlim A \otimes I_{n}) \ar[r] \ar[d]^{\cong} & K_{1}(\varinjlim A \otimes D_{n}) \ar[d]^{\cong}\\
	 K_{0}(A \otimes \varinjlim I_{n}) \ar[r]  & K_{0}(A \otimes \varinjlim D_{n}) \ar[r] & K_{0}(A \otimes \varinjlim B_{n}) \ar[r] & K_{1}(A \otimes \varinjlim I_{n}) \ar[r] & K_{1}(A \otimes \varinjlim D_{n})
	}
\end{equation*}
By five-lemma, the map $\beta_{*}$ is an isomorphism.

($\Leftarrow$): 
Conversely, let $A$ be a $K$-continuous \cast-algebra. Then the functor $F(B) \coloneqq K_{0}(A \otimes B)$, considered on the category of separable \cast-algebras, factors through the asymptotic homotopy category of Connes-Higson by \cite[Theorem 3.11]{MR1262931}. Since $F$ is stable and satisfies Bott periodicity, it in fact factors through the category $E$. Hence by Corollary~\ref{cor factor}, $A$ is $K$-exact.
\end{proof}


\bibliographystyle{amsalpha}
\bibliography{../../MathPapers/biblio}

\def\romsup#1{{\edef\next{\the\font}$^{\next#1}$}} \def\cprime{$'$}
  \def\cftil#1{\ifmmode\setbox7\hbox{$\accent"5E#1$}\else
  \setbox7\hbox{\accent"5E#1}\penalty 10000\relax\fi\raise 1\ht7
  \hbox{\lower1.15ex\hbox to 1\wd7{\hss\accent"7E\hss}}\penalty 10000
  \hskip-1\wd7\penalty 10000\box7} \def\Dbar{\leavevmode\lower.6ex\hbox to
  0pt{\hskip-.23ex \accent"16\hss}D}
  \def\cfac#1{\ifmmode\setbox7\hbox{$\accent"5E#1$}\else
  \setbox7\hbox{\accent"5E#1}\penalty 10000\relax\fi\raise 1\ht7
  \hbox{\lower1.15ex\hbox to 1\wd7{\hss\accent"13\hss}}\penalty 10000
  \hskip-1\wd7\penalty
  10000\box7}\def\cfudot#1{\ifmmode\setbox7\hbox{$\accent"5E#1$}\else
  \setbox7\hbox{\accent"5E#1}\penalty 10000\relax\fi\raise 1\ht7
  \hbox{\raise.1ex\hbox to 1\wd7{\hss.\hss}}\penalty 10000 \hskip-1\wd7\penalty
  10000\box7} \def\polhk#1{\setbox0=\hbox{#1}{\ooalign{\hidewidth
  \lower1.5ex\hbox{`}\hidewidth\crcr\unhbox0}}}
  \def\cydot{\leavevmode\raise.4ex\hbox{.}}
  \def\cfgrv#1{\ifmmode\setbox7\hbox{$\accent"5E#1$}\else
  \setbox7\hbox{\accent"5E#1}\penalty 10000\relax\fi\raise 1\ht7
  \hbox{\lower1.05ex\hbox to 1\wd7{\hss\accent"12\hss}}\penalty 10000
  \hskip-1\wd7\penalty 10000\box7}
  \def\uarc#1{\ifmmode{\lineskiplimit=0pt\oalign{$#1$\crcr
  \hidewidth\setbox0=\hbox{\lower1ex\hbox{{\rm\char"15}}}\dp0=0pt
  \box0\hidewidth}}\else{\lineskiplimit=0pt\oalign{#1\crcr
  \hidewidth\setbox0=\hbox{\lower1ex\hbox{{\rm\char"15}}}\dp0=0pt
  \box0\hidewidth}}\relax\fi}
\providecommand{\bysame}{\leavevmode\hbox to3em{\hrulefill}\thinspace}
\providecommand{\MR}{\relax\ifhmode\unskip\space\fi MR }
\providecommand{\MRhref}[2]{%
  \href{http://www.ams.org/mathscinet-getitem?mr=#1}{#2}
}
\providecommand{\href}[2]{#2}
\begin{thebibliography}{CEOO04}

\bibitem[AD09]{MR2562137}
Claire Anantharaman-Delaroche, \emph{On tensor products of group
  {$C^\ast$}-algebras and related topics}, Limits of graphs in group theory and
  computer science, EPFL Press, Lausanne, 2009, pp.~1--35. \MR{2562137
  (2011a:46082)}

\bibitem[Bla98]{MR1656031}
Bruce Blackadar, \emph{{$K$}-theory for operator algebras}, second ed.,
  Mathematical Sciences Research Institute Publications, vol.~5, Cambridge
  University Press, Cambridge, 1998. \MR{1656031 (99g:46104)}

\bibitem[BO08]{MR2391387}
Nathanial~P. Brown and Narutaka Ozawa, \emph{{$C^*$}-algebras and
  finite-dimensional approximations}, Graduate Studies in Mathematics, vol.~88,
  American Mathematical Society, Providence, RI, 2008. \MR{2391387
  (2009h:46101)}

\bibitem[CEOO04]{MR2100669}
J.~Chabert, S.~Echterhoff, and H.~Oyono-Oyono, \emph{Going-down functors, the
  {K}{\"u}nneth formula, and the {B}aum-{C}onnes conjecture}, Geom. Funct.
  Anal. \textbf{14} (2004), no.~3, 491--528. \MR{2100669 (2005h:19005)}

\bibitem[CH90]{MR1065438}
Alain Connes and Nigel Higson, \emph{D\'eformations, morphismes asymptotiques
  et {$K$}-th\'eorie bivariante}, C. R. Acad. Sci. Paris S\'er. I Math.
  \textbf{311} (1990), no.~2, 101--106. \MR{1065438 (91m:46114)}

\bibitem[D{\u{a}}d94]{MR1262931}
Marius D{\u{a}}d{\u{a}}rlat, \emph{Shape theory and asymptotic morphisms for
  {$C^*$}-algebras}, Duke Math. J. \textbf{73} (1994), no.~3, 687--711.
  \MR{1262931 (95c:46117)}

\bibitem[EH85]{MR791294}
Edward~G. Effros and Uffe Haagerup, \emph{Lifting problems and local
  reflexivity for {$C\sp \ast$}-algebras}, Duke Math. J. \textbf{52} (1985),
  no.~1, 103--128. \MR{791294 (86k:46084)}

\bibitem[Hig90]{MR1068250}
Nigel Higson, \emph{Categories of fractions and excision in {$KK$}-theory}, J.
  Pure Appl. Algebra \textbf{65} (1990), no.~2, 119--138. \MR{1068250
  (91i:19005)}

\bibitem[HLS02]{MR1911663}
N.~Higson, V.~Lafforgue, and G.~Skandalis, \emph{Counterexamples to the
  {B}aum-{C}onnes conjecture}, Geom. Funct. Anal. \textbf{12} (2002), no.~2,
  330--354. \MR{1911663 (2003g:19007)}

\bibitem[Kir83]{MR715549}
Eberhard Kirchberg, \emph{The {F}ubini theorem for exact {$C^{\ast}
  $}-algebras}, J. Operator Theory \textbf{10} (1983), no.~1, 3--8. \MR{715549
  (85d:46081)}

\bibitem[Kir95]{MR1322641}
\bysame, \emph{On subalgebras of the {CAR}-algebra}, J. Funct. Anal.
  \textbf{129} (1995), no.~1, 35--63. \MR{1322641 (95m:46094b)}

\bibitem[Oza03]{MR1964549}
Narutaka Ozawa, \emph{An application of expanders to {$\Bbb B(l_2)\otimes\Bbb
  B(l_2)$}}, J. Funct. Anal. \textbf{198} (2003), no.~2, 499--510. \MR{1964549
  (2004d:46065)}

\bibitem[RLL00]{MR1783408}
M.~R{\o}rdam, F.~Larsen, and N.~Laustsen, \emph{An introduction to {$K$}-theory
  for {$C\sp *$}-algebras}, London Mathematical Society Student Texts, vol.~49,
  Cambridge University Press, Cambridge, 2000. \MR{1783408 (2001g:46001)}

\bibitem[Ros82]{MR658514}
Jonathan Rosenberg, \emph{The role of {$K$}-theory in noncommutative algebraic
  topology}, Operator algebras and $K$-theory (San Francisco, Calif., 1981),
  Contemp. Math., vol.~10, Amer. Math. Soc., Providence, R.I., 1982,
  pp.~155--182. \MR{658514 (84h:46097)}

\bibitem[Sch82]{MR650021}
Claude Schochet, \emph{Topological methods for {$C\sp{\ast} $}-algebras. {II}.
  {G}eometric resolutions and the {K}{\"u}nneth formula}, Pacific J. Math.
  \textbf{98} (1982), no.~2, 443--458. \MR{650021 (84g:46105b)}

\bibitem[Ska91]{MR1143449}
Georges Skandalis, \emph{Le bifoncteur de {K}asparov n'est pas exact}, C. R.
  Acad. Sci. Paris S\'er. I Math. \textbf{313} (1991), no.~13, 939--941.
  \MR{1143449 (93b:46136)}

\bibitem[Tak02]{MR1873025}
M.~Takesaki, \emph{Theory of operator algebras. {I}}, Encyclopaedia of
  Mathematical Sciences, vol. 124, Springer-Verlag, Berlin, 2002, Reprint of
  the first (1979) edition, Operator Algebras and Non-commutative Geometry, 5.
  \MR{1873025 (2002m:46083)}

\bibitem[Voi90]{MR1026768}
Dan Voiculescu, \emph{Property {$T$} and approximation of operators}, Bull.
  London Math. Soc. \textbf{22} (1990), no.~1, 25--30. \MR{1026768 (90m:46101)}

\bibitem[Was90]{MR1058315}
Simon Wassermann, \emph{Tensor products of free-group {$C^*$}-algebras}, Bull.
  London Math. Soc. \textbf{22} (1990), no.~4, 375--380. \MR{1058315
  (91h:46103)}

\bibitem[Wei94]{MR1269324}
Charles~A. Weibel, \emph{An introduction to homological algebra}, Cambridge
  Studies in Advanced Mathematics, vol.~38, Cambridge University Press,
  Cambridge, 1994. \MR{1269324 (95f:18001)}

\bibitem[WO93]{MR1222415}
N.~E. Wegge-Olsen, \emph{{$K$}-theory and {$C\sp *$}-algebras}, Oxford Science
  Publications, The Clarendon Press Oxford University Press, New York, 1993, A
  friendly approach. \MR{1222415 (95c:46116)}

\end{thebibliography}
\end{document}